\documentclass[a4paper,11pt]{article}  
\usepackage{amsmath}
\usepackage{amssymb}
\usepackage{theorem}
\usepackage{pstricks}
\usepackage{xcolor}
\title{\sffamily Moreau's Decomposition in Banach Spaces}
\author{Patrick L. Combettes$^1$
and Noli N. Reyes$^2$ \\[5mm]
$\!^1$UPMC Universit\'e Paris 06\\
Laboratoire Jacques-Louis Lions -- UMR CNRS 7598\\
75005 Paris, France\\
{\ttfamily plc@math.jussieu.fr}\\[4mm]
$^2$University of the Philippines -- Diliman\\
Institute of Mathematics\\
Quezon City, 1101 Philippines\\
{\ttfamily noli@math.upd.edu.ph}}
\date{~}
\newcommand{\pair}[2]{\langle{{#1},{#2}}\rangle}
\newcommand{\bpair}[2]{\left\langle{{#1},{#2}}\right\rangle}
\newcommand{\menge}[2]{\big\{{#1}~\big |~{#2}\big\}} 
\newcommand{\Menge}[2]{\bigg\{{#1}~\bigg |~{#2}\bigg\}} 
\newcommand{\XX}{\ensuremath{\mathcal X}}
\newcommand{\XP}{\ensuremath{\mathcal X}^*}
\newcommand{\RX}{\ensuremath{\,\left]-\infty,+\infty\right]}}
\newcommand{\RXX}{\ensuremath{\,\left[-\infty,+\infty\right]}}

\newcommand{\IDD}{\ensuremath{\operatorname{int}\operatorname{dom}f}}
\newcommand{\IDP}{\ensuremath{\operatorname{int}\operatorname{dom}f^*}}

\newcommand{\intdom}{\ensuremath{\operatorname{int}\operatorname{dom}}\,}
\newcommand{\inte}{\ensuremath{\operatorname{int}}}
\newcommand{\dom}{\ensuremath{\operatorname{dom}}}

\newcommand{\Nf}{\ensuremath{\nabla f}}

\newcommand{\sri}{\ensuremath{{\operatorname{sri}}\,}}

\newcommand{\prox}{\ensuremath{{\operatorname{prox}}}}
\newcommand{\aprox}{\ensuremath{{\operatorname{aprox}}}}
\newcommand{\bprox}{\ensuremath{{\operatorname{bprox}}}}

\newcommand{\spc}{\ensuremath{\overline{\operatorname{span}}\,}}
\newcommand{\emp}{\ensuremath{{\varnothing}}}
\newcommand{\Id}{\ensuremath{\operatorname{Id}}\,}
\newcommand{\RR}{\ensuremath{\mathbb{R}}}

\newcommand{\infconv}{\ensuremath{\mbox{\footnotesize$\,\square\,$}}}
\newcommand{\ldiamond}{\ensuremath{\mbox{\Large$\,\diamond\,$}}}
\newcommand{\RPP}{\ensuremath{\left]0,+\infty\right[}}

\newcommand{\exi}{\ensuremath{\exists\,}}
\newcommand{\pinf}{\ensuremath{{+\infty}}}
\newcommand{\minf}{\ensuremath{{-\infty}}}

\newtheorem{theorem}{Theorem}[section]
\newtheorem{lemma}[theorem]{Lemma}

\newtheorem{corollary}[theorem]{Corollary}
\newtheorem{proposition}[theorem]{Proposition}
\theoremstyle{plain}{\theorembodyfont{\rmfamily}%
}
\theoremstyle{plain}{\theorembodyfont{\rmfamily}%
}
\theoremstyle{plain}{\theorembodyfont{\rmfamily}%
}
\theoremstyle{plain}{\theorembodyfont{\rmfamily}%
\newtheorem{remark}[theorem]{Remark}}
\theoremstyle{plain}{\theorembodyfont{\rmfamily}%
\newtheorem{definition}[theorem]{Definition}}
\theoremstyle{plain}{\theorembodyfont{\rmfamily}%
}

\numberwithin{equation}{section}
\begin{document}
\maketitle

\begin{abstract}
Moreau's decomposition is a powerful nonlinear hilbertian analysis 
tool that has been used in various areas of optimization and
applied mathematics. 
In this paper, it is extended to reflexive Banach spaces and in 
the context of generalized proximity measures. This extension 
unifies and significantly improves upon existing results.
\end{abstract}

\vskip 11mm

\newpage

\section{Introduction}

Throughout this paper, $(\XX,\|\cdot\|)$ is a reflexive real Banach 
space with topological dual $(\XP,\|\cdot\|_*)$, and the canonical 
bilinear form on $\XX\times\XP$ is denoted by $\pair{\cdot}{\cdot}$. 
The distance function to a set $C\subset\XX$ is
$d_C\colon x\mapsto\inf_{y\in C}\|x-y\|$, the metric projector onto
$C$ is $P_C\colon x\mapsto\menge{y\in C}{\|x-y\|=d_C(x)}$, and the 
polar cone of $C$ is $C^\ominus=\menge{x^*\in\XX^*}
{(\forall x\in C)\;\pair{x}{x^*}\leq 0}$.
$\Gamma_0(\XX)$ is the class of lower semicontinuous convex 
functions $\varphi\colon\XX\to\RX$ such that
$\dom\varphi=\menge{x\in\XX}{\varphi(x)<\pinf}\neq\emp$.

A classical tool in linear hilbertian analysis is the following
orthogonal decomposition principle.

\begin{proposition}
\label{p:1}
Suppose that $\XX$ is a Hilbert space, let $V$ be a closed 
vector subspace of $\XX$ with orthogonal complement $V^\bot$, and 
let $x\in\XX$. Then the following hold.
\begin{enumerate}
\item
$\|x\|^2=d_V^2(x)+d_{V^\bot}^2(x)$.
\item
$x=P_Vx+P_{V^\bot}x$.
\item
$\pair{P_Vx}{P_{V^\perp}x}=0$.
\end{enumerate}
\end{proposition}

In 1962, Moreau proposed a nonlinear extension of this decomposition.

\begin{proposition} {\rm \cite{Mor62a}}
\label{p:2}
Suppose that $\XX$ is a Hilbert space, let $K$ be a nonempty 
closed convex cone in $\XX$, and let $x\in\XX$. Then the following 
hold.
\begin{enumerate}
\item
\label{p:2ii}
$\|x\|^2=d_K^2(x)+d_{K^\ominus}^2(x)$.
\item
\label{p:2i}
$x=P_Kx+P_{K^\ominus}x$.
\item
\label{p:2iii}
$\pair{P_Kx}{P_{K^\ominus}x}=0$.
\end{enumerate}
\end{proposition}

Motivated by problems in unilateral mechanics, Moreau further
extended this result in \cite{Mor62b} (see also \cite{More65}). 
To state Moreau's decomposition principle, we require some basic 
notions from convex analysis \cite{Livre1,Zali02}. Let 
$\varphi$ and $f$ be two functions in $\Gamma_0(\XX)$. 
The conjugate of $\varphi$ is the function 
$\varphi^*$ in $\Gamma_0(\XX^*)$ defined by 
\begin{equation}
\label{e:elnido2011-03-02e}
\varphi^*\colon\XX^*\to\RX\colon x^*\mapsto
\sup_{x\in\XX}\big(\pair{x}{x^*}-\varphi(x)\big). 
\end{equation}
Moreover, the infimal convolution of $\varphi$ and $f$ is
the function 
\begin{equation}
\label{e:elnido2011-03-02f}
\varphi\infconv f\colon\XX\to\RXX\colon
x\mapsto\inf_{y\in\XX}\big(\varphi(y)+f(x-y)\big).
\end{equation}
Now suppose that $\XX$ is a Hilbert space and set
$q=(1/2)\|\cdot\|^2$. Then, for every $x\in\XX$, there exists
a unique point $p\in\XX$ such that 
$(\varphi\infconv q)(x)=\varphi(p)+q(x-p)$; this point is denoted
by $p=\prox_\varphi x$. 
The operator $\prox_\varphi\colon\XX\to\XX$ thus defined is 
called the proximity operator of $\varphi$. 

\begin{proposition}{\rm \cite{Mor62b,More65}}
\label{p:3}
Suppose that $\XX$ is a Hilbert space, let
$\varphi\in\Gamma_0(\XX)$, set $q=\|\cdot\|^2/2$, and let 
$x\in\XX$. Then the following hold.
\begin{enumerate}
\item
$q(x)=(\varphi\infconv q)(x)+(\varphi^*\infconv q)(x)$.
\item
\label{p:3i}
$x=\prox_\varphi x+\prox_{\varphi^*}x$.
\item
$\pair{\prox_\varphi x}{\prox_{\varphi^*}x}
=\varphi\big(\prox_\varphi x\big)+
\varphi^*\big(\prox_{\varphi^*}x\big)$.
\end{enumerate}
\end{proposition}

Note that, if in Proposition~\ref{p:3} $\varphi$ is the indicator 
function of a nonempty closed convex cone $K\subset\XX$, i.e.,
$\varphi=\iota_K$ where
\begin{equation}
\label{e:iota}
(\forall x\in\XX)\quad\iota_K(x)=
\begin{cases}
0,&\text{if}\;\;x\in K;\\
\pinf,&\text{if}\;\;x\notin K,
\end{cases}
\end{equation}
we recover Proposition~\ref{p:2}.

The above hilbertian nonlinear decomposition principles have 
found many applications in optimization and in various other
areas of applied mathematics (see for instance 
\cite{Brog06,Coll79,Svva10,Siop07,Smms05,Hans84,Jbhu89,Hiri10,%
Luce10,Rock06} and the references therein)
and attempts have been made to extend them to more general Banach 
spaces. The main result in this direction is the following
generalization of Proposition~\ref{p:2}\ref{p:2i}\&\ref{p:2iii} 
in uniformly convex and uniformly smooth Banach spaces
(see also \cite{Albe05,Huyh11,Scho08,Song04} for alternate proofs 
and applications), where $\Pi_C$ denotes the generalized projector
onto a nonempty closed convex subset $C$ of $\XX$ \cite{Albe96}, 
i.e., if $J$ denotes the duality mapping of $\XX$,
\begin{equation}
\label{e:gproj}
(\forall x\in\XX)\quad\Pi_C
x=\underset{y\in C}{\operatorname{argmin}}
\big(\|x\|^2-2\pair{y}{Jx}+\|y\|^2\big).
\end{equation}

\begin{proposition}{\rm \cite{Albe00}}
\label{p:4}
Suppose that $\XX$ is uniformly convex and uniformly smooth,
let $J\colon\XX\to\XX^*$ denote its duality mapping, which is
characterized by
\begin{equation}
\label{e:dualitymap}
(\forall x\in\XX)\quad\|x\|^2=\pair{x}{Jx}=\|Jx\|^2_*,
\end{equation}
let $K$ be a nonempty closed convex cone in $\XX$, and let
$x\in\XX$. Then the following hold.
\begin{enumerate}
\item
\label{p:4i}
$x=P_Kx+ J^{-1}\big(\Pi_{K^\ominus}(Jx)\big)$.
\item
\label{p:4iii}
$\pair{P_Kx}{\Pi_{K^\ominus}(Jx)}=0$.
\end{enumerate}
\end{proposition}

The objective of the present paper is to unify and extend the
above results. To this end, we first discuss in Section~\ref{sec:2}
suitable notions of proximity in Banach spaces. Based on these, we
propose our extension of Moreau's decomposition in
Section~\ref{sec:3}. A feature of our analysis is to rely heavily
on convex analytical tools, which allows us to derive our main
result with simpler proofs than those utilized in the above
special case.

\section{Proximity in Banach spaces}
\label{sec:2}

Let $\varphi\in\Gamma_0(\XX)$. 
As seen in the Introduction, if $\XX$ is a Hilbert space, Moreau's 
proximity operator is defined by 
\begin{equation}
\label{e:2011-03-12a}
(\forall x\in\XX)\quad\prox_\varphi x=\underset{y\in\XX}
{\operatorname{argmin}}\Big(\varphi(y)+\frac12\|x-y\|^2\Big).
\end{equation}
In this section we discuss two extensions of this operator
in Banach spaces. We recall that $\varphi$ is coercive if
$\lim_{\|y\|\to\pinf}\varphi(y)=\pinf$ and supercoercive if 
$\lim_{\|y\|\to\pinf}\varphi(y)/\|y\|=\pinf$. 
As usual, the subdifferential operator of $\varphi$ is denoted by 
$\partial\varphi$. Finally, the strong relative interior of a 
convex set $C\subset\XX$ is
\begin{equation} 
\label{e:dsri}
\sri C=\Menge{x\in C}{\bigcup_{\lambda>0}\lambda(C-x)=\spc (C-x)}.
\end{equation}
We shall also require the following facts.

\begin{lemma}[\rm\cite{More64,Rock66}]
\label{l:7}
Let $f\in\Gamma_0(\XX)$ and let $x^*\in\XX^*$. Then $f-x^*$ is 
coercive if and only if $x^*\in\intdom f^*$.
\end{lemma}

\begin{lemma}[\rm{\cite[Theorem~3.4]{Ccm101}}]
\label{l:8}
Let $f\in\Gamma_0(\XX)$ be supercoercive. Then 
$\dom f^*=\XX^*$.
\end{lemma}

\begin{lemma}[\rm\cite{Atto86}]
\label{l:9}
Let $f$ and $\varphi$ be functions in $\Gamma_0(\XX)$ such that 
$0\in\sri(\dom f-\dom\varphi)$. Then the following hold.
\begin{enumerate}
\item
\label{l:9i}
$(\varphi+f)^*=\varphi^*\infconv f^*$ and the infimal convolution 
is exact everywhere: $(\forall x^*\in\XX^*)(\exi y^*\in\XX^*)$
$(\varphi+f)^*(x^*)=\varphi^*(y^*)+f^*(x^*-y^*)$.
\item
\label{l:9ii}
$\partial(\varphi+f)=\partial\varphi+\partial f$.
\end{enumerate}
\end{lemma}

\subsection{Legendre functions}
We review the notion of a Legendre function, which was 
introduced in Euclidean spaces in \cite{Rock70} and extended to 
Banach spaces in \cite{Ccm101} (see also \cite{Borw01} for 
further developments in the nonreflexive case).

\begin{definition}{\rm\cite[Definition~5.2]{Ccm101}}
\label{d:legendre}
Let $f\in\Gamma_0(\XX)$. Then $f$ is:
\begin{enumerate}
\item 
essentially smooth, if $\partial f$ is both locally
bounded and single-valued on its domain;
\item 
essentially strictly convex, if $(\partial f)^{-1}$ is locally 
bounded on its domain and $f$ is strictly convex on every
convex subset of $\dom\partial f$;
\item 
a Legendre function, if it is both essentially smooth and
essentially strictly convex.
\end{enumerate}
\end{definition}

Some key properties of Legendre functions are listed below.

\begin{lemma}
\label{l:5}
Let $f\in\Gamma_0(\XX)$ be a Legendre function. 
Then the following hold.
\begin{enumerate}
\item
\label{l:5i}
$f^*$ is a Legendre function {\rm\cite[Corollary~5.5]{Ccm101}}.
\item
\label{l:5ii}
$\dom\partial f=\IDD\neq\emp$ and $f$ is G\^ateaux differentiable 
on $\IDD$ {\rm\cite[Theorem~5.6]{Ccm101}}.
\item
\label{l:5iii}
$\nabla f\colon\IDD\to\IDP$ is bijective with inverse
$\nabla f^*\colon\IDP\to\IDD$ {\rm\cite[Theorem~5.10]{Ccm101}}.
\end{enumerate}
\end{lemma}

\subsection{$D$-proximity operators}

In this subsection we discuss a notion of proximity based on
Bregman distances investigated in \cite{Sico03} and which goes back
to \cite{Cens92,Tebo92}.

The first extension of \eqref{e:2011-03-12a} was investigated in 
\cite{Sico03}. Let $f\in\Gamma_0(\XX)$ be a Legendre function. 
The Bregman distance associated with $f$ is
\begin{equation}
\label{e:elnido2011-03-05}
\begin{aligned}
D_f\colon\XX\times\XX&\to\,[0,\pinf]\\
(y,x)&\mapsto 
\begin{cases}
f(y)-f(x)-\pair{y-x}{\Nf(x)},&\text{if}\;\;x\in\IDD;\\
\pinf,&\text{otherwise}.
\end{cases}
\end{aligned}
\end{equation}
For every $\varphi\in\Gamma_0(\XX)$, we define the function
$\varphi\ldiamond f\colon\XX\to\RXX$ by
\begin{equation}
\label{e:inf-convo-breg} 
(\forall x\in\XX)\quad
(\varphi\ldiamond f)(x)=\inf_{y\in\XX}\big(\varphi(y)+
D_f(y,x)\big).
\end{equation}

The following proposition refines and complements some results of
\cite[Section~3.4]{Sico03}.

\begin{proposition}
\label{p:2011-04-13}
Let $f\in\Gamma_0(\XX)$ be a Legendre function, let
$\varphi\in\Gamma_0(\XX)$ be such that 
\begin{equation}
\label{e:2011-04-13a}
0\in\sri(\dom f-\dom\varphi),
\end{equation}
and let $x\in\intdom f$.
Suppose that one of the following holds.
\begin{enumerate}
\item
\label{p:2011-04-13i}
$\Nf(x)\in\inte(\dom f^*+\dom\varphi^*)$.
\item
\label{p:2011-04-13ii}
$\intdom f^*\subset\inte(\dom f^*+\dom\varphi^*)$.
\item
\label{p:2011-04-13iii}
$f$ is supercoercive.
\item
\label{p:2011-04-13iv}
$\inf\varphi(\XX)>\minf$.
\end{enumerate}
Then there exists a unique point $p\in\XX$ such that 
$(\varphi\ldiamond f)(x)=\varphi(p)+D_f(p,x)$; 
moreover, $p$ lies in $\dom\partial\varphi\cap\IDD$ and it is 
characterized by the inclusion
\begin{equation}
\label{e:2011-04-13b}
\Nf(x)-\Nf(p)\in\partial\varphi(p).
\end{equation}
\end{proposition}
\begin{proof}
Set $f_x\colon\XX\to\RX\colon y\mapsto f(y)-\pair{y}{\Nf(x)}$. 
Then the minimizers of $\varphi+D_f(\cdot,x)$ coincide with those 
of $\varphi+f_x$ and our assumptions imply that 
\begin{equation}
\label{e:2011-04-13e}
\varphi+f_x\in\Gamma_0(\XX). 
\end{equation}
Now let $p\in\XX$. 
It follows from \eqref{e:2011-04-13a}, Lemma~\ref{l:9}\ref{l:9ii},
and Lemma~\ref{l:5}\ref{l:5ii} that
\begin{eqnarray}
(\varphi\ldiamond f)(x)=\varphi(p)+D_f(p,x)
&\Leftrightarrow&p\;\:\text{minimizes}\;\:\varphi+f_x
\nonumber\\
&\Leftrightarrow&0\in\partial\big(\varphi+f_x\big)(p)
\nonumber\\
&\Leftrightarrow&0\in\partial\varphi(p)+\partial f(p)-\Nf(x)
\nonumber\\
&\Leftrightarrow&0\in\partial\varphi(p)+\Nf(p)-\Nf(x)
\nonumber\\
&\Leftrightarrow&\Nf(x)-\Nf(p)\in\partial\varphi(p)
\label{e:2011-04-13c}\\
&\Rightarrow&p\in\dom\partial\varphi\cap\IDD.
\label{e:2011-04-13d}
\end{eqnarray}
Hence, the minimizers of $\varphi+f_x$ are in $\IDD$. However,
since $f$ is essentially strictly convex, it is strictly convex 
on $\IDD$ and so is therefore $\varphi+f_x$. This shows
that $\varphi+f_x$ admits at most one minimizer.
It remains to establish existence. 

\ref{p:2011-04-13i}: 
It follows from \eqref{e:2011-04-13e} that, to show existence, it 
is enough to show that $\varphi+f_x$ is coercive 
\cite[Theorem~2.5.1(ii)]{Zali02}. In view of Lemma~\ref{l:7}, this
is equivalent to showing that $\Nf(x)\in\intdom(f+\varphi)^*$.
However, it follows from \eqref{e:2011-04-13a} and 
Lemma~\ref{l:9}\ref{l:9i} that
\begin{equation}
\label{e:2011-04-13f}
\intdom(f+\varphi)^*
=\intdom(f^*\infconv\varphi^*)
=\inte(\dom f^*+\dom\varphi^*).
\end{equation}

\ref{p:2011-04-13ii}$\Rightarrow$\ref{p:2011-04-13i}: 
Lemma~\ref{l:5}\ref{l:5iii}.

\ref{p:2011-04-13iii}$\Rightarrow$\ref{p:2011-04-13ii}: 
By Lemma~\ref{l:8}, $\dom f^*=\XX^*$ and, since 
$\dom\varphi^*\neq\emp$, 
$\intdom f^*\subset\inte(\dom f^*+\dom\varphi^*)$.

\ref{p:2011-04-13iv}$\Rightarrow$\ref{p:2011-04-13ii}: 
We have 
$\inf\varphi(\XX)>\minf$ $\Rightarrow$
$\varphi^*(0)=-\inf\varphi(\XX)<\pinf$ $\Rightarrow$ 
$0\in\dom\varphi^*$. Hence,
$\intdom f^*\subset\inte(\dom f^*+\dom\varphi^*)$.
\end{proof}

In view of Proposition~\ref{p:2011-04-13} and
Lemma~\ref{l:5}\ref{l:5iii}, the following is well 
defined.

\begin{definition}
\label{d:bprox}
Let $f\in\Gamma_0(\XX)$ be a Legendre function and let
$\varphi\in\Gamma_0(\XX)$ be such that
$0\in\sri(\dom f-\dom\varphi)$. Set 
\begin{equation}
E=(\IDD)\cap\big(\Nf^*\big(\inte(\dom f^*+\dom\varphi^*)\big)\big).
\end{equation}
The $D$-proximity (or Bregman proximity) operator of $\varphi$ 
relative to $f$ is 
\begin{equation}
\label{e:avril2011a}
\begin{aligned}
\bprox_\varphi^f\colon E\to\IDD\colon
x\mapsto\underset{y\in\XX}
{\operatorname{argmin}}\big(\varphi(y)+D_f(y,x)\big).
\end{aligned}
\end{equation}
\end{definition}

\begin{remark}
\label{r:2011-03-13a}
In connection with Definition~\ref{d:bprox}, let us make a couple
of observations.
\begin{enumerate}
\item
It follows from Proposition~\ref{p:2011-04-13} that,
if $\intdom f^*\subset\inte(\dom\varphi^*+\dom f^*)$ (in particular
if $f$ is supercoercive or if $\inf\varphi(\XX)>\minf$), then
$\bprox_\varphi^f\colon\IDD\to\IDD$.
\item
\label{r:2011-03-13aii}
Suppose that $\XX$ is hilbertian and that $f=\|\cdot\|^2/2$, and 
let $\varphi\in\Gamma_0(\XX)$. 
Then $\varphi\ldiamond f=\varphi\infconv f$ and 
$\bprox_\varphi^f=\prox_\varphi$.
\end{enumerate}
\end{remark}

\subsection{Anisotropic proximity operators}

An alternative extension of the notion of proximity can be 
obtained by replacing the function $\|\cdot\|^2/2$ in  
\eqref{e:2011-03-12a} by a Legendre function $f$.

\begin{proposition}
\label{p:2011-04-14}
Let $f\in\Gamma_0(\XX)$ be a Legendre function, let
$\varphi\in\Gamma_0(\XX)$ be such that
\begin{equation}
\label{e:2011-04-14b}
0\in\sri(\dom f^*-\dom\varphi^*),
\end{equation}
and let $x\in\sri(\dom f+\dom\varphi)$. 
Then there exists a unique point $p\in\XX$ such that 
$(\varphi\infconv f)(x)=\varphi(p)+f(x-p)$; moreover, $p$ is 
characterized by the inclusion
\begin{equation}
\label{e:2011-03-11}
\Nf(x-p)\in\partial\varphi(p).
\end{equation}
\end{proposition}
\begin{proof}
Using \eqref{e:2011-04-14b} and Lemma~\ref{l:9}\ref{l:9i}, we obtain
\begin{equation}
(\varphi^*+f^*)^*=\varphi^{**}\infconv f^{**}=\varphi\infconv f
\end{equation}
and the fact that the infimum in the infimal convolution is 
attained everywhere. 
On the other hand, since $x\in\sri(\dom f+\dom\varphi)$, we have
\begin{equation}
0\in\sri\big(\dom\varphi-(x-\dom f)\big)
=\sri\big(\dom\varphi-\dom f(x-\cdot)\big). 
\end{equation}
Consequently, by Lemma~\ref{l:9}\ref{l:9ii},
\begin{equation}
\label{e:sum-rule}
\partial\big(\varphi+f(x-\cdot)\big)
=\partial\varphi+\partial f(x-\cdot). 
\end{equation}
Now let $p\in\XX$. It follows from \eqref{e:sum-rule} and
Lemma~\ref{l:5}\ref{l:5ii} that
\begin{eqnarray}
p\;\:\text{minimizes}\;\:\varphi+f(x-\cdot)
&\Leftrightarrow&0\in\partial\big(\varphi+f(x-\cdot)\big)(p)
\nonumber\\
&\Leftrightarrow&0\in\partial\varphi(p)-\partial f(x-p)
\nonumber\\
&\Leftrightarrow&0\in\partial\varphi(p)-\Nf(x-p)
\nonumber\\
&\Leftrightarrow&\Nf(x-p)\in\partial\varphi(p)
\label{e:2011-04-26a}\\
&\Rightarrow&x-p\in\IDD.
\label{e:2011-04-26b}
\end{eqnarray}
To show uniqueness, suppose that $p$ and $q$ are two distinct
minimizers of $\varphi+f(x-\cdot)$. Then
$(\varphi\infconv f)(x)=\varphi(p)+f(x-p)=\varphi(q)+f(x-q)$ and,
by \eqref{e:2011-04-26b}, $x-p$ and $x-q$ lie in $\IDD$.
Now let $r=(1/2)(p+q)$ and suppose that $p\neq q$. 
Lemma~\ref{l:5}\ref{l:5ii} asserts that $f$ is strictly 
convex on the convex set $\IDD=\dom\partial f$.  Therefore,
invoking the convexity of $\varphi$,
\begin{align}
(\varphi\infconv f)(x)
&\leq\varphi(r)+f(x-r)\nonumber\\
&<\frac12\big(\varphi(p)+\varphi(q)\big)+
\frac12\big(f(x-p)+f(x-q)\big)\nonumber\\
&=(\varphi\infconv f)(x),
\end{align}
which is impossible.
\end{proof}

Using Proposition~\ref{p:2011-04-14}, we can now introduce the 
anisotropic proximity operator of $\varphi$.

\begin{definition}
\label{d:aprox}
Let $f\in\Gamma_0(\XX)$ be a Legendre function and let
$\varphi\in\Gamma_0(\XX)$ be such that
$0\in\sri(\dom f^*-\dom\varphi^*)$. Set 
\begin{equation}
E=\sri(\dom f+\dom\varphi).
\end{equation}
The anisotropic proximity operator of $\varphi$ 
relative to $f$ is 
\begin{equation}
\label{e:avril2011b}
\aprox_\varphi^f\colon E\to\XX\colon
x\mapsto\underset{y\in\XX}{\operatorname{argmin}}
\big(\varphi(y)+f(x-y)\big).
\end{equation}
\end{definition}

\begin{remark}
\label{r:2011-03-13b}
Suppose that $\XX$ is hilbertian and that $f=\|\cdot\|^2/2$, and 
let $\varphi\in\Gamma_0(\XX)$. Then 
$\aprox_\varphi^f=\prox_\varphi$.
\end{remark}

\section{Main result}
\label{sec:3}

In the previous section we have described two extensions of the
classical proximity operator. Our main result is a generalization 
of Moreau's decomposition (Proposition~\ref{p:3}) in Banach 
spaces which involves a mix of these two extensions.

\begin{theorem}
\label{t:puerto-princesa2011-03-06}
Let $f\in\Gamma_0(\XX)$ be a Legendre function, let
$\varphi\in\Gamma_0(\XX)$ be such that
\begin{equation}
\label{e:2011-04-15b}
0\in\sri(\dom f^*-\dom\varphi^*),
\end{equation}
and let $x\in(\IDD)\cap\inte(\dom f+\dom\varphi)$. 
Then the following hold.
\begin{enumerate}
\item
\label{t:puerto-princesa2011-03-06i}
$f(x)=(\varphi\infconv f)(x)+(\varphi^*\ldiamond f^*)
\big(\Nf(x)\big)$.
\item
\label{t:puerto-princesa2011-03-06iia}
$x=\aprox^f_\varphi x+\nabla f^*
\big(\bprox^{f^*}_{\varphi^*}\big(\nabla f(x)\big)\big)$.
\item
\label{t:puerto-princesa2011-03-06iib}
$\bpair{\aprox^f_\varphi x}{\bprox_{\varphi^*}^{f^*}
\big(\Nf(x)\big)}=\varphi\big(\aprox^f_\varphi x\big)+
\varphi^*\big(\bprox_{\varphi^*}^{f^*}\big(\Nf(x)\big)\big)$.
\item
\label{t:puerto-princesa2011-03-06iic}
$\bpair{\aprox^f_\varphi x}{\Nf\big(x-\aprox^f_\varphi x\big)}=
\varphi\big(\aprox^f_\varphi x\big)+
\varphi^*\big(\Nf\big(x-\aprox^f_\varphi x\big)\big)$.
\end{enumerate}
\end{theorem}
\begin{proof}
Since $x\in\inte(\dom f+\dom\varphi)$,
Lemma~\ref{l:5}\ref{l:5iii} yields
\begin{equation}
x\in\sri(\dom f+\dom\varphi)\quad\text{and}\quad
\Nf^*\big(\Nf(x)\big)\in\inte\big(\dom f^{**}+\dom\varphi^{**}\big).
\end{equation}
Hence, it follows from Proposition~\ref{p:2011-04-14} 
that $\aprox^f_\varphi x$ is well defined and, from
Lemma~\ref{l:5}\ref{l:5i} and 
Proposition~\ref{p:2011-04-13}\ref{p:2011-04-13i}
(applied to $f^*$ and $\varphi^*$), that 
$\Nf^*(\bprox_{\varphi^*}^{f^*}(\Nf(x)))$ is well defined.
In addition,
\begin{equation}
\label{e:2011-04-27}
(\varphi\infconv f)(x)\in\RR
\quad\text{and}\quad
(\varphi^*\ldiamond f^*)\big(\Nf(x)\big)\in\RR.
\end{equation}

\ref{t:puerto-princesa2011-03-06i}:
It follows from \eqref{e:elnido2011-03-05}, 
Lemma~\ref{l:5}\ref{l:5iii}, and the Fenchel-Young identity 
\cite[Theorem~2.4.2(iii)]{Zali02} that
\begin{align} 
(\forall x^*\in\XX^*)\quad
D_{f^*}\big(x^*,\Nf(x)\big)  
&=f^*( x^*)-f^*\big(\Nf(x)\big)-\pair{x^*-\Nf(x)}{x}_*\nonumber\\
&=f^*(x^*)+f(x)-\pair{x^*}{x}_*.
\end{align}
This, \eqref{e:inf-convo-breg}, \eqref{e:2011-04-15b}, and 
Lemma~\ref{l:9}\ref{l:9i} imply that 
\begin{align}
(\varphi^*\ldiamond f^*)\big(\Nf (x)\big) 
&=\inf_{x^*\in\XX^*}\big(\varphi^*(x^*)+f^*(x^*)+f(x)-
\pair{x^*}{x}_*\big)\nonumber\\
&=f(x)-\sup_{x^*\in\XX^*}\big(\pair{x^*}{x}_*-\varphi^*(x^*)-
f^*(x^*)\big)\nonumber\\
&=f(x)-(\varphi^*+f^*)^*(x)\nonumber\\
&=f(x)-(\varphi\infconv f)(x).
\end{align}
In view of \eqref{e:2011-04-27}, we obtain the announced identity.

\ref{t:puerto-princesa2011-03-06iia}:
Let $p\in\XX$. Using Proposition~\ref{p:2011-04-14},
Lemma~\ref{l:5}\ref{l:5iii}, and
Proposition~\ref{p:2011-04-13}\ref{p:2011-04-13i}, we obtain
\begin{eqnarray}
p=\aprox^f_\varphi x
&\Leftrightarrow&\Nf(x-p)\in\partial\varphi(p)
\label{e:puerto-princesa2011-03-06a}\\
&\Leftrightarrow&
p\in\partial\varphi^*\big(\Nf(x-p)\big)\nonumber\\
&\Leftrightarrow&\Nf^*\big(\Nf(x)\big)-\Nf^*\big(\Nf(x-p)\big)\in
\partial\varphi^*\big(\Nf(x-p)\big)\nonumber\\
&\Leftrightarrow&\Nf(x-p)=\bprox_{\varphi^*}^{f^*}\big(\Nf(x)\big)
\label{e:puerto-princesa2011-03-06b}\\
&\Leftrightarrow&x-p=\Nf^*\big(\bprox_{\varphi^*}^{f^*}
\big(\Nf(x)\big)\big).
\end{eqnarray}

\ref{t:puerto-princesa2011-03-06iib}:
Set $p=\aprox^f_\varphi x$. As seen in 
\eqref{e:puerto-princesa2011-03-06b} and
\eqref{e:puerto-princesa2011-03-06a},
\begin{equation}
\label{e:puerto-princesa2011-03-06c}
\bprox_{\varphi^*}^{f^*}\big(\Nf(x)\big)=\Nf(x-p)\in
\partial\varphi(p).
\end{equation}
Hence, the Fenchel-Young identity yields
\begin{align}
\label{e:2010-11-04h}
\pair{p}{\bprox_{\varphi^*}^{f^*}\big(\Nf(x)\big)}
&=\pair{p}{\Nf(x-p)}\nonumber\\
&=\varphi(p)+\varphi^*\big(\Nf(x-p)\big)\nonumber\\
&=\varphi(p)+\varphi^*\big(\bprox_{\varphi^*}^{f^*}
\big(\Nf(x)\big)\big).
\end{align}

\ref{t:puerto-princesa2011-03-06iic}: This follows at once from
\ref{t:puerto-princesa2011-03-06iib} and 
\eqref{e:puerto-princesa2011-03-06c}.
\end{proof}

Theorem~\ref{t:puerto-princesa2011-03-06} provides a range of new
decomposition schemes, even in the case when $\XX$ is a Hilbert 
space. Thus, in the following result, we obtain a new hilbertian
frame decomposition principle (for background on frames and their 
applications, see \cite{Chri08}). 

\begin{corollary}
\label{c:frames}
Suppose that $\XX$ is a separable Hilbert space, let $I$ be a
countable set, and let $(e_i)_{i\in I}$ be a frame in $\XX$, i.e., 
\begin{equation}
\label{e:frame-bounds}
(\exi\alpha\in\RPP)(\exi\beta\in\RPP)(\forall x\in\XX)\quad
\alpha\|x\|^2\leq\sum_{i\in I}|\pair{x}{e_i}|^2\leq\beta\|x\|^2.
\end{equation}
Let $S\colon\XX\to\XX\colon x\mapsto\sum_{i\in I}
\pair{x}{e_i}e_i$ be the associated frame operator and let
$(e^*_i)_{i\in I}=(S^{-1}e_i)_{i\in I}$ be the associated 
canonical dual frame. Furthermore, let $\varphi\in\Gamma_0(\XX)$,
let $x\in\XX$, and set 
\begin{equation}
\label{e:frame2}
a(x)=\underset{y\in \XX}{\operatorname{argmin}} 
\left(\varphi(y)+\frac{1}{2}\sum_{i\in I}|\pair{x-y}{e_i}|^2
\right)
\end{equation}
and
\begin{equation}
\label{e:frame3}
b(x)=\underset{x^*\in\XX}{\operatorname{argmin}}
\left(\varphi^*(x^*)-\pair{x^*}{x}+\frac{1}{2}\sum_{i\in I}
|\pair{x^*}{e^*_i}|^2\right).
\end{equation}
Then $x=a(x)+\sum_{i\in I}\pair{b(x)}{e^*_i}e^*_i$.
\end{corollary}
\begin{proof}
Set $f\colon\XX\to\RR\colon x\mapsto(1/2)
\sum_{i\in I}|\pair{x}{e_i}|^2$. It is easily seen that $f$ is 
Fr\'echet differentiable on $\XX$ with $\Nf=S$. It therefore
follows from \cite[Theorem~5.6]{Ccm101} that $f$ is essentially
smooth. Now fix $x^*\in\XX$. Since the frame operator of 
$(e_i^*)_{i\in I}$ is $S^{-1}$ \cite[Lemma~5.1.6]{Chri08}, we have
\begin{equation}
\label{e:2011-04-18c}
\pair{S^{-1}x^*}{x^*}=
\bpair{\sum_{i\in I}\pair{x^*}{e_i^*}e_i^*}{x^*}
=\sum_{i\in I}|\pair{x^*}{e_i^*}|^2=2f(S^{-1}x^*).
\end{equation}
Now set $g\colon\XX\to\RR\colon x\mapsto f(x)-\pair{x}{x^*}$.
Then $g$ is a differentiable convex function and 
$\nabla g\colon x\mapsto Sx-x^*$ vanishes at $x=S^{-1}x^*$.
Hence, using \eqref{e:2011-04-18c}, we obtain
\begin{equation}
\label{e:conju-frame}
f^*(x^*)=-\min_{x\in\XX}g(x)=\pair{S^{-1}x^*}{x^*}-f(S^{-1}x^*)
=f(S^{-1}x^*)=\frac12\sum_{i\in I}|\pair{x^*}{e^*_i}|^2.
\end{equation}
Hence, as above, $f^*$ is Fr\'echet differentiable on $\XX$ with 
$\Nf^*=S^{-1}$ and, in turn, essentially smooth, which makes $f$ 
essentially strictly convex \cite[Theorem~5.4]{Ccm101}. Altogether,
$f$ is a Legendre function with 
\begin{equation}
\label{e:2011-04-18d}
\dom f=\XX,\quad\dom f^*=\XX,\quad\Nf=S,
\quad\text{and}\quad\Nf^*=S^{-1}.
\end{equation}
Moreover, 
it follows from \eqref{e:avril2011a}, \eqref{e:avril2011b}, 
\eqref{e:2011-04-18d}, Lemma~\ref{l:5}\ref{l:5iii}, \eqref{e:frame2},
\eqref{e:frame3}, and \eqref{e:conju-frame} that
\begin{equation}
\bprox_{\varphi^*}^{f^*}(\Nf(x))=b(x)
\quad\text{and}\quad
\aprox_{\varphi}^{f}(x)=a(x).
\end{equation}
The result is therefore an application of 
Theorem~\ref{t:puerto-princesa2011-03-06}%
\ref{t:puerto-princesa2011-03-06iia}. 
\end{proof}

\begin{remark}
\label{r:2011-04-18}
Corollary~\ref{c:frames} can be regarded as an extension of Moreau's
decomposition principle in separable Hilbert spaces. Indeed, in the
special case when $(e_i)_{i\in I}$ is an orthonormal basis in
Corollary~\ref{c:frames}, we recover Proposition~\ref{p:3}\ref{p:3i}.
\end{remark}

The next application is set in uniformly convex and uniformly smooth
Banach spaces.
\begin{corollary}
\label{c:puerto-princesa2011-03-06}
Suppose that $\XX$ is uniformly convex and uniformly smooth, 
let $J$ be its duality mapping, set $q=\|\cdot\|^2/2$, and 
let $\varphi\in\Gamma_0(\XX)$.
Then $q^*=\|\cdot\|^2_*/2$ and the following hold for every 
$x\in\XX$.
\begin{enumerate}
\item
\label{c:puerto-princesa2011-03-06ii}
$q(x)=(\varphi\infconv q)(x)+(\varphi^*\ldiamond q^*)(Jx)$.
\item
\label{c:puerto-princesa2011-03-06i}
$x=\aprox^q_\varphi x+
J^{-1}\big(\bprox^{q^*}_{\varphi^*}(Jx)\big)$.
\item
\label{c:puerto-princesa2011-03-06iii}
$\bpair{\aprox^q_\varphi x}{\bprox_{\varphi^*}^{q^*}(Jx)}=
\varphi\big(\aprox^q_\varphi x\big)+
\varphi^*\big(\bprox_{\varphi^*}^{q^*}(Jx)\big)$.
\item
\label{c:puerto-princesa2011-03-06iv}
$\bpair{\aprox^q_\varphi x}{J\big(x-\aprox^q_\varphi x\big)}=
\varphi\big(\aprox^q_\varphi x\big)+
\varphi^*\big(J\big(x-\aprox^q_\varphi x\big)\big)$.
\end{enumerate}
\end{corollary}
\begin{proof}
This is an application of Theorem~\ref{t:puerto-princesa2011-03-06} 
with $f=q$. Indeed, $\dom f=\XX$, $\dom f^*=\XX^*$, and $\Nf=J$.
\end{proof}

In particular, if $\XX$ is a Hilbert space in 
Corollary~\ref{c:puerto-princesa2011-03-06}, if follows from
Remark~\ref{r:2011-03-13a}\ref{r:2011-03-13aii} and 
Remark~\ref{r:2011-03-13b} that we 
recover Moreau's decomposition principle (Proposition~\ref{p:3}) and
a fortiori Propositions~\ref{p:1} and \ref{p:2}.
Another noteworthy instance of 
Corollary~\ref{c:puerto-princesa2011-03-06}
is when $\varphi=\iota_K$, where $K$ is a nonempty closed convex 
cone in $\XX$. In this case, $\varphi^*=\iota_{K^\ominus}$, 
$\aprox^q_\varphi=P_K$, and we derive
from \eqref{e:gproj} and \eqref{e:dualitymap} that 
$\bprox^q_\varphi=\Pi_K$. Hence,
Corollary~\ref{c:puerto-princesa2011-03-06}%
\ref{c:puerto-princesa2011-03-06i}\&%
\ref{c:puerto-princesa2011-03-06iii} yields Proposition~\ref{p:4}.

\begin{remark}
\label{r:2011-03-16}
Consider the setting of 
Theorem~\ref{t:puerto-princesa2011-03-06} 
and set $A=\partial\varphi$. Then, 
by Rockafellar's theorem, $A$ is a maximally monotone operator
\cite[Theorem~3.1.11]{Zali02}. Moreover, it follows from 
\eqref{e:2011-03-11}, Lemma~\ref{l:5}\ref{l:5iii},
and \eqref{e:2011-04-13b} that we can rewrite 
Theorem~\ref{t:puerto-princesa2011-03-06}%
\ref{t:puerto-princesa2011-03-06iia} as 
\begin{equation}
\label{e:2011-03-16j}
x=(\Id+\nabla f^*\circ A)^{-1}x+
\nabla f^*\circ\big(\Nf^*+A^{-1}\big)^{-1}x,
\end{equation}
where $\Id$ is the identity operator on $\XX$.
The results of \cite[Section~3.3]{Sico03} suggest that this
decomposition holds for more general maximally monotone
operators $A\colon\XX\to 2^{\XX^*}$. 
If $\XX$ is a Hilbert space and
$f=\|\cdot\|^2/2$, \eqref{e:2011-03-16j} yields the well-known
resolvent identity $\Id=(\Id+A)^{-1}+(\Id+A^{-1})^{-1}$,
which is true for any maximally monotone operator $A$
\cite[Proposition~23.18]{Livre1}.
\end{remark}

\noindent
{\bfseries Acknowledgement.} 
The work of the first author was supported by the Agence 
Nationale de la Recherche under grant ANR-08-BLAN-0294-02.
The authors thank one of the referees for making some 
valuable comments.

\end{document}